 \documentclass[a4paper,12pt]{article}
    \usepackage[top=2cm,bottom=2cm,left=2cm,right=2cm]{geometry}
    \usepackage{cite, amsmath, amssymb}
    \usepackage[margin=1cm,%
                font=small,%
                format=hang,%
                labelsep=period,%
                labelfont=bf]{caption}
\usepackage[latin1]{inputenc}
\usepackage{amsmath}
\usepackage{amsfonts}
\usepackage{amssymb}
\usepackage{cite}
\usepackage{amsthm}
\usepackage{makeidx}
\usepackage{graphicx}
\usepackage{mathrsfs,xcolor}
\usepackage{enumerate}
\usepackage{blkarray}
\usepackage{bm}
\usepackage{setspace}
\usepackage{float}
\usepackage{authblk}

\numberwithin{table}{section}
\numberwithin{equation}{section}

\theoremstyle{plain}
\newtheorem{theorem}{Theorem}[section]
\newtheorem{proposition}[theorem]{Proposition}
\newtheorem{definition}[theorem]{Definition}

\newtheorem{example}[theorem]{Example}

\newtheorem{remark}[theorem]{Remark}

\usepackage{authblk}
\author[1,*]{ \textbf{Bryan S. Hernandez}}
 \author[1]{\textbf{Ralph John L. De la Cruz}} 

\affil[1]{\small \textit{Institute of Mathematics, University of the Philippines Diliman, Quezon City 1101, Philippines }}
\affil[*]{Corresponding author: \texttt{bryan.hernandez@upd.edu.ph}}

\title{\textbf{Independent Decompositions of Chemical Reaction Networks}}

\date{\normalsize (2020)}

\begin{document}
\maketitle
\thispagestyle{empty}
\begin{abstract}
	A chemical reaction network (CRN) is composed of reactions that can be seen as interactions among entities called species, which exist within the system. Endowed with kinetics, CRN has a corresponding set of ordinary differential equations (ODEs). In Chemical Reaction Network Theory, we are interested with connections between the structure of the CRN and qualitative properties of the corresponding ODEs. One of the results in Decomposition Theory of CRNs is that the intersection of the sets of positive steady states of the subsystems is equal to the set of positive steady states of the whole system, if the decomposition is independent. Hence, computational approach using independent decompositions can be used as an efficient tool in studying large systems.
In this work, we provide a necessary and sufficient condition for the existence of a nontrivial independent decomposition of a CRN, which leads to a novel step-by-step method to obtain such decomposition, if it exists. We also illustrate these results using real-life examples. In particular, we show that a CRN of a popular model of anaerobic yeast fermentation pathway has a nontrivial independent decomposition, while a particular biological system, which is a metabolic network with one positive feedforward and a negative feedback has none. Finally, we analyze properties of steady states of reaction networks of specific influenza virus models.
\end{abstract}
\baselineskip=0.30in
\section{Introduction}
\label{sec:1}
A chemical reaction network (CRN) is composed of reactions which can be seen as interactions among species existing within the system.
A CRN endowed with a chemical kinetics is called a chemical kinetic system (CKS). This system has a corresponding set of ordinary differential equations that describes its dynamics over time. Our goal in studying Chemical Reaction Network Theory (CRNT) is to explore relationships between the network structure and qualitative properties of the corresponding differential equations.
From the past decades, several papers considered CRNT in studying dynamical systems. It was presumed to be a promising framework for understanding of complex biological systems based on structural properties \cite{interferon}. In 2015, Arceo et al. provided a framework to represent a Biochemical Systems Theory (BST) model as a CRN endowed with power-law kinetics \cite{arceo2015}. 
Recent results involve the possibility of multistationarity, i.e., existence of multiple steady states for particular rate constants, in very complex continuous flow stirred tank reactors and to the possibility of traveling composition waves on isothermal catalyst surfaces \cite{feinberg:some}. In fact, CRNT is also used to study multistationarity in models of carbon cycle and biochemical processes \cite{fortun2,hmr2019}.

One among the important results in decomposition theory in CRNT relates the set of steady states of the whole system to the sets of steady states of its subsystems under a decomposition induced by partitioning the reaction set, which we call the Feinberg Decomposition Theorem. In general, the intersection of the sets of positive steady states of these subsystems is contained in the set of positive steady states of the system itself. If the decomposition of the underlying network is independent, i.e., the stoichiometric subspace of the network is equal to the direct sum of the stoichiometric subspaces of the subnetworks, the equality holds. The equality is important, in the sense that, under an independent decomposition, if a system has a subsystem which does not have the capacity for multistationarity, then so is the whole network. Hence, finding independent decompositions of a CRN, if they exist, leads to efficiency in computation. Decomposition theory, in particular independent ones, has potential and may be used to infer properties of larger networks from smaller ones. Previous studies on independent decompositions on a particular type called ``fundamental decompositions'' and those decompositions existing in literature were considered in \cite{bsh1} and \cite{bsh2}, respectively. In this work, we will focus on any decomposition and see whether there is such an independent decomposition existing in a network.

%In this paper, we will also deal with the problem of multistationarity of power-law kinetic systems, a generalization of mass action kinetic systems.
In 2020, Hernandez et al. \cite{hmr2019} provided a Multistationarity Algorithm for power-law kinetic systems, i.e., CRNs with power-law kinetics, which is an extension of the Higher Deficiency Algorithm of Ji and Feinberg \cite{ji} for mass action systems. From the name of the algorithm itself, it solves the problem of determining whether a power-law system has the capacity for multistationarity for particular constants within a stoichiometric class.

Given a CRN, we aim to do the following:
\begin{itemize}
	\item[1.] Provide conditions when a nontrivial independent decomposition of a CRN exists or does not exist.
	\item[2.] Construct a method that gives an independent decomposition, if it exists.
	\item[3.] Provide real-life examples to illustrate the preceding method such as influenza virus models and other areas.
	\item[4.] Relate the obtained results with existing deficiency theorems and multistationarity algorithms.
\end{itemize}

This paper is organized as follows. In Section \ref{prelim}, we provide important preliminaries as the fundamentals of CRNs and CKS, and discussion of deficiency theorems and decomposition theory. Section \ref{sec:results} deals with our main results and how these are applied to real-life examples and integrate with existing theorems and algorithms in the literature.
Finally, Section \ref{sec:sum} gives a summary and outlook for future directions.

\section{Preliminaries}
%\section{Fundamentals of Chemical Reaction Networks and Kinetic Systems}
\label{prelim}
\indent 
In this section, we consider important notions about chemical reaction networks and chemical kinetic systems \cite{arceo2015,feinberg,wiuf}. We also present results on the decomposition theory, which was introduced by Feinberg in \cite{feinberg12,feinberg:book}.

\subsection{Fundamentals of Chemical Reaction Networks}

\begin{definition}
	A {\bf chemical reaction network} (CRN) is the triple $\mathscr{N} = \left(\mathscr{S},\mathscr{C},\mathscr{R}\right)$ of the nonempty finite sets $\mathscr{S}$, $\mathscr{C} \subseteq \mathbb{R}_{\ge 0}^\mathscr{S}$, and $\mathscr{R} \subset \mathscr{C} \times \mathscr{C}$, of $m$ species, $n$ complexes, and $r$ reactions, respectively, that satisfy: $\left( {{C_i},{C_i}} \right) \notin \mathscr{R}$ for each $C_i \in \mathscr{C}$; and
	for each $C_i \in \mathscr{C}$, there exists $C_j \in \mathscr{C}$ such that $\left( {{C_i},{C_j}} \right) \in \mathscr{R}$ or $\left( {{C_j},{C_i}} \right) \in \mathscr{R}$.
\end{definition}
We can view $\mathscr{C}$ as a subset of $\mathbb{R}^m_{\ge 0}$. The ordered pair $\left( {{C_i},{C_j}} \right)$ corresponds to the familiar notation ${C_i} \to {C_j}$.

{
	\begin{definition}
		The {\bf molecularity matrix}, denoted by $Y$, is an $m\times n$ matrix such that $Y_{ij}$ is the stoichiometric coefficient of the species $X_i$ in complex $C_j$.
		The {\bf incidence matrix} $I_a$ is an $n\times r$ matrix such that 
		$${\left( {{I_a}} \right)_{ij}} = \left\{ \begin{array}{rl}
			- 1&{\rm{ if \ }}{C_i}{\rm{ \ is \ in \ the\ reactant \ of \ reaction \ }}{R_j},\\
			1&{\rm{  if \ }}{C_i}{\rm{ \ is \ in \ the\ product \ of \ reaction \ }}{R_j},\\
			0&{\rm{    otherwise}}.
		\end{array} \right.$$
		The {\bf stoichiometric matrix}, denoted by $N$, is the $m\times r$ matrix given by 
		$N=YI_a$.
	\end{definition}
}

\begin{definition}
	The {\bf reaction vectors} for a given reaction network $\left(\mathscr{S},\mathscr{C},\mathscr{R}\right)$ are the elements of the set $\left\{{C_j} - {C_i} \in \mathbb{R}^m|\left( {{C_i},{C_j}} \right) \in \mathscr{R}\right\}.$
\end{definition}

\begin{definition}
	The {\bf stoichiometric subspace} of a reaction network $\left(\mathscr{S},\mathscr{C},\mathscr{R}\right)$, denoted by $S$, is the linear subspace of $\mathbb{R}^m$ given by $$S = {\text{span}}\left\{ {{C_j} - {C_i} |\left( {{C_i},{C_j}} \right) \in \mathscr{R}} \right\}.$$ The {\bf rank} of the network, denoted by $s$, is given by $s=\dim S$. The set $\left( {x + S} \right) \cap \mathbb{R}_{ \ge 0}^m$ is said to be a {\bf stoichiometric compatibility class} of $x \in \mathbb{R}_{ \ge 0}^m$.
\end{definition}

\begin{definition}
	Two vectors $x, x^{*} \in {\mathbb{R}^m}$ are {\bf stoichiometrically compatible} if $x-x^{*}$ is an element of the stoichiometric subspace $S$.
\end{definition}

CRNs can be seen as directed graphs. One can view complexes as vertices and reactions as edges.
%At this point, if we are talking about geometric properties, {\bf vertices} are complexes and {\bf edges} are reactions.
If there is a path between two vertices $C_i$ and $C_j$, then they are said to be {\bf connected}. If there is a directed path from vertex $C_i$ to vertex $C_j$ and vice versa, then they are said to be {\bf strongly connected}. If any two vertices of a subgraph are {\bf (strongly) connected}, then the subgraph is said to be a {\bf (strongly) connected component}. The (strong) connected components are precisely the {\bf (strong) linkage classes} of a CRN. The maximal strongly connected subgraphs where there are no edges from a complex in the subgraph to a complex outside the subgraph is said to be the {\bf terminal strong linkage classes}.
We denote the number of linkage classes and the number of strong linkage classes by $l$ and $sl$, respectively.
A CRN is said to be {\bf weakly reversible} if $sl=l$.

\begin{definition}
	The {\bf deficiency} of a CRN, denoted by $\delta$, is given by $$\delta=n-l-s$$
	where $n$ is the number of complexes, $l$ is the number of linkage classes, and $s$ is the dimension of the stoichiometric subspace $S$.
\end{definition}

\subsection{Fundamentals of Chemical Kinetic Systems}

\begin{definition}
	A {\bf kinetics} $K$ for a reaction network $\left(\mathscr{S},\mathscr{C},\mathscr{R}\right)$ is an assignment to each reaction $r: y \to y' \in \mathscr{R}$ of a rate function ${K_r}:{\Omega _K} \to {\mathbb{R}_{ \ge 0}}$ such that $\mathbb{R}_{ > 0}^m \subseteq {\Omega _K} \subseteq \mathbb{R}_{ \ge 0}^m$, $c \wedge d \in {\Omega _K}$ if $c,d \in {\Omega _K}$, and ${K_r}\left( c \right) \ge 0$ for each $c \in {\Omega _K}$.
	Furthermore, it satisfies the positivity property: supp $y$ $\subset$ supp $c$ if and only if $K_r(c)>0$.
	The system $\left(\mathscr{S},\mathscr{C},\mathscr{R},K\right)$ is called a {\bf chemical kinetic system}.
\end{definition}

\begin{definition}
	The {\bf species formation rate function} (SFRF) of a chemical kinetic system is given by $f\left( x \right) = NK(x)= \displaystyle \sum\limits_{{C_i} \to {C_j} \in \mathscr{R}} {{K_{{C_i} \to {C_j}}}\left( x \right)\left( {{C_j} - {C_i}} \right)}.$
\end{definition}

The ordinary differential equation (ODE) or dynamical system of a chemical kinetic system is $\dfrac{{dx}}{{dt}} = f\left( x \right)$. A zero of $f$ is called a {\bf equilibrium} or {\bf steady state}.

\begin{definition}
	The {\bf set of positive steady states} of a chemical kinetic system $\left(\mathscr{S},\mathscr{C},\mathscr{R},K\right)$ is given by ${E_ + }\left(\mathscr{S},\mathscr{C},\mathscr{R},K\right)= \left\{ {x \in \mathbb{R}^m_{>0}|f\left( x \right) = 0} \right\}.$
\end{definition}

A CRN is said to admit {\bf multiple (positive) steady states} if there exist positive rate constants such that the ODE system admits more than one stoichiometrically compatible steady states.

\begin{definition}
	A kinetics $K$ is a {\bf power-law kinetics} (PLK) if 
	${K_i}\left( x \right) = {k_i}{{x^{{F_{i}}}}} $ for $i =1,...,r$ where ${k_i} \in {\mathbb{R}_{ > 0}}$ and ${F_{ij}} \in {\mathbb{R}}$. The power-law kinetics is identified with an $r \times m$ matrix $F$, called the {\bf kinetic order matrix} and a vector $k \in \mathbb{R}^r$, called the {\bf rate vector}.
\end{definition}
The system becomes the well-known {\bf mass action kinetics (MAK)} if the kinetic order matrix is the transpose of the molecularity matrix.

%\begin{definition}
%A PLK system has {\bf reactant-determined kinetics} (of type PL-RDK) if for any two reactions $i, j$ with identical reactant complexes, the corresponding rows of kinetic orders in $F$ are identical, i.e., ${F_{ik}} = {F_{jk}}$ for $k = 1,2,...,m$. A PLK system has {\bf non-reactant-determined kinetics} (of type PL-NDK) if there exist two reactions with the same reactant complexes whose corresponding rows in $F$ are not identical.
%\end{definition}

\subsection{Deficiency Theorems and Decomposition Theory}
\label{sect:decomposition}

We now state the Deficiency Zero and Deficiency One Theorems primarily from the works of Feinberg \cite{feinberg12,feinberg,feinberg2}.

\begin{theorem} (Deficiency Zero Theorem)
	For any CRN of deficiency zero, the following statements hold:
	\begin{enumerate}
		\item[i.] If the network is not weakly reversible, then for arbitrary kinetics, the differential equations for the corresponding reaction system cannot admit
		a positive steady state.
		\item[ii.] If the network is not weakly reversible, then for arbitrary kinetics, the differential equations for the corresponding reaction system cannot admit
		a cyclic composition trajectory containing a positive composition.
		\item[iii.] If the network is weakly reversible, then { for mass action kinetics} (regardless
		of the positive values the rate constants take), the differential equations have these properties:\\
		There exists within each positive stoichiometric compatibility class precisely one steady state; that steady state is asymptotically stable; and there cannot exist a nontrivial cyclic composition
		trajectory along which all species concentrations are positive.
	\end{enumerate}
\end{theorem}

\begin{theorem} (Deficiency One Theorem)
	Consider a mass action system. Let $\delta$ be the deficiency of the network
	and let $\delta_\theta$ be the deficiency of the $\theta$th linkage class, each containing just one terminal strong linkage class. Suppose that both of the following
	conditions hold:
	\begin{enumerate}
		\item[i.] $\delta_\theta \le 1$ for each linkage class and
		\item[ii.] the sum of the deficiencies of all the individual linkage classes equals the deficiency of the whole network.
	\end{enumerate}
	Then, no matter what positive values the rate constants
	take, the corresponding differential equations can admit
	no more than one steady state within a positive stoichiometric compatibility class.
	If the network is weakly reversible, the differential equations for the system admit
	precisely one steady state in each positive stoichiometric compatibility class.
\end{theorem}

We consider definitions and earlier results from the decomposition theory of chemical reaction networks. 

\begin{definition}
	A {\textbf{decomposition}} of $\mathscr{N}$ is a set of subnetworks $\{\mathscr{N}_1, \mathscr{N}_2,...,\mathscr{N}_k\}$ of $\mathscr{N}$ induced by a partition $\{\mathscr{R}_1, \mathscr{R}_2,...,\mathscr{R}_k\}$ of its reaction set $\mathscr{R}$. 
\end{definition}

We denote a decomposition by 
$\mathscr{N} = \mathscr{N}_1 \cup \mathscr{N}_2 \cup ... \cup \mathscr{N}_k$
as $\mathscr{N}$ is a union of the subnetworks
%in the sense of
\cite{ghms2019}. It also follows immediately that, for the corresponding stoichiometric subspaces, 
${S} = {S}_1 + {S}_2 + \cdots + {S}_k$.

{
	A network decomposition $\mathscr{N} = \mathscr{N}_1 \cup \mathscr{N}_2 \cup ... \cup \mathscr{N}_k$  is a {\bf refinement} of
	$\mathscr{N} = {\mathscr{N}'}_1 \cup {\mathscr{N}'}_2 \cup ... \cup {\mathscr{N}'}_{k'}$ 
	(and the latter a {\bf coarsening} of the former) if it is induced by a refinement  
	$\{\mathscr{R}_1, \mathscr{R}_2,...,\mathscr{R}_k\}$
	of $\{{\mathscr{R}'}_1 \cup {\mathscr{R}'}_2 \cup ... \cup {\mathscr{R}'}_{k'}\}$.
	
	\begin{example}
		Consider $\mathscr{N} = \mathscr{N}_1 \cup \mathscr{N}_2 \cup \mathscr{N}_3$ where the subnetworks are provided in Table \ref{network:refine:coarse}. Let $\mathscr{N}' = \mathscr{N}_2 \cup \mathscr{N}_3$. Then the decomposition $\mathscr{N} = \mathscr{N}_1 \cup \mathscr{N}_2 \cup \mathscr{N}_3$ of $\mathscr{N}$ is a refinement of $\mathscr{N} = \mathscr{N}_1 \cup \mathscr{N}'$. In addition, $\mathscr{N} = \mathscr{N}_1 \cup \mathscr{N}'$ is a coarsening of $\mathscr{N} = \mathscr{N}_1 \cup \mathscr{N}_2 \cup \mathscr{N}_3$.
		\label{ex:refine:coarse}
	\end{example}
	
}

\begin{table}
	\centering
	\caption{Subnetworks in Example \ref{ex:refine:coarse}}
	\label{network:refine:coarse}       % Give a unique label
	\begin{tabular}{lc}
		%& $\mathscr{N}$  & $\mathscr{N}_1$ & $\mathscr{N}_2$\\
		\noalign{\smallskip}\hline\noalign{\smallskip}
		$\mathscr{N}_1$ & $0 \mathbin{\lower.3ex\hbox{$\buildrel\textstyle\rightarrow\over
				{\smash{\leftarrow}\vphantom{_{\vbox to.5ex{\vss}}}}$}} A$\\
		$\mathscr{N}_2$ & $A + B \to C$\\
		$\mathscr{N}_3$ & $C \to 0$\\
		\noalign{\smallskip}\hline\noalign{\smallskip}
	\end{tabular}
\end{table}

The following important concept of independent decomposition was introduced by Feinberg in \cite{feinberg12}.

{
	
	\begin{definition}
		A network decomposition is said to be {\textbf{independent}} if its stoichiometric subspace is equal to the direct sum of the subnetwork stoichiometric subspaces.
		On the otherhand, it is said to be {\textbf{incidence independent}} if its incidence map (incidence matrix) is equal to the direct sum of the incidence maps (incidence matrices) of the subnetworks.
	\end{definition}
	
	A network $\mathscr{N}$ has at least one independent decomposition given by $\{\mathscr{N}\}$. We call this decomposition as the {\bf{trivial independent decomposition}} of $\mathscr{N}$. In other words, there is only one subnetwork and it is the network itself. Correspondingly, a nontrivial independent decomposition yields at least two subnetworks.
}

\begin{table}
	\begin{center}
		\caption{Details for computation in Example \ref {ex:computation}}
		\label{computation}       % Give a unique label
		\begin{tabular}{lcccc}
			\noalign{\smallskip}\hline\noalign{\smallskip}
			& stoichiometric matrix  & $s$ or $s_i$ & incidence matrix $I_a$ or $I_{a,i}$& rank of $I_a$ or $I_{a,i}$\\
			\noalign{\smallskip}\hline\noalign{\smallskip}
			$\mathscr{N}$ & $\left[ {\begin{array}{*{20}{c}}
					1&{ - 1}&0&0\\
					0&1&0&0\\
					0&0&{ - 1}&0\\
					0&0&1&0\\
					0&0&1&{ - 1}
			\end{array}} \right]$ & 4 & $\left[ {\begin{array}{*{20}{c}}
					{ - 1}&0&0&1\\
					1&{ - 1}&0&0\\
					0&1&0&0\\
					0&0&{ - 1}&0\\
					0&0&1&0\\
					0&0&0&{ - 1}
			\end{array}} \right]$ & 4\\
			$\mathscr{N}_1$ & $\left[ {\begin{array}{*{20}{c}}
					1&{ - 1}\\
					0&1
			\end{array}} \right]$ & 2 & $\left[ {\begin{array}{*{20}{c}}
					{ - 1}&0\\
					1&{ - 1}\\
					0&1
			\end{array}} \right]$ & 2\\
			$\mathscr{N}_2$ & $\left[ {\begin{array}{*{20}{c}}
					{ - 1}&0\\
					1&0\\
					1&{ - 1}
			\end{array}} \right]$ & 2 & $\left[ {\begin{array}{*{20}{c}}
					{ - 1}&0\\
					1&0\\
					0&{ - 1}\\
					0&1
			\end{array}} \right]$ & 2\\
			\noalign{\smallskip}\hline\noalign{\smallskip}
		\end{tabular}
	\end{center}
\end{table}

{
	\begin{example} 
		Consider a network $\mathscr{N}$ with the following reactions:
		\[\begin{array}{l}
			{R_1}:0 \to {X_1}\\
			{R_2}:{X_1} \to {X_2}\\
			{R_3}:{X_3} \to {X_4} + {X_5}\\
			{R_4}:{X_5} \to 0
		\end{array}\]
		and its decomposition with subnetworks $\mathscr{N}_1=\{R_1,R_2\}$ and $\mathscr{N}_2=\{R_3,R_4\}$. In Table \ref{computation}, we can check that the sum of the dimensions of the stoichiometric subspaces (or ranks of the stoichiometric matrices) of the subnetworks equals that of the whole network, i.e., $4=2+2$. Hence, the decomposition is independent. Analogously, the sum of the ranks of the incidence matrices of the subnetworks equals that of the whole network, i.e., $4=2+2$. Therefore, the decomposition is also incidence independent.
		\label{ex:computation}	
	\end{example}
}

Feinberg established the following relation between an independent decomposition and the set of positive steady states of a kinetic system.

\begin{theorem} (Feinberg Decomposition Theorem \cite{feinberg12})
	\label{feinberg:decom:thm}
	Let $P(\mathscr{R})=\{\mathscr{R}_1, \mathscr{R}_2,...,\mathscr{R}_k\}$ be a partition of a CRN $\mathscr{N}$ and let $K$ be a kinetics on $\mathscr{N}$. If %$\mathscr{N} = \mathscr{N}_1 \cup \mathscr{N}_2 \cup ... \cup \mathscr{N}_k$
	$\mathscr{N} = \bigcup\limits_{i = 1}^k \mathscr{N}_k$
	is the network decomposition of $P(\mathscr{R})$ and ${E_ + }\left(\mathscr{N}_i,{K}_i\right)= \left\{ {x \in \mathbb{R}^\mathscr{S}_{>0}|N_iK_i(x) = 0} \right\}$ then
	%\[{E_ + }\left(\mathscr{N}_1,K_1\right) \cap {E_ + }\left(\mathscr{N}_2,K_2\right) \cap ... \cap {E_ + }\left(\mathscr{N}_k,K_k\right) \subseteq  {E_ + }\left(\mathscr{N},K\right).\]
	\[\bigcap\limits_{i = 1}^k {E_ + }\left(\mathscr{N}_i,K_i\right) \subseteq  {E_ + }\left(\mathscr{N},K\right).\]
	If the network decomposition is independent, then equality holds.
\end{theorem}

{The theorem given above is a powerful tool in studying properties of equilibria of a CRN in terms of studying the properties of its subnetworks.}

\section{Results and Applications}
\label{sec:results}
{
	In this work, we basically deal with independent decomposition given a chemical reaction network. With this assumption of independence, the Feinberg Decomposition Theorem (FDT) gives an equation in terms of the intersection of the sets of equilibria of the subnetworks and that of the whole network.
	In case we use decompositions which are not independent, then the stated equality does not eventually hold.
	Hence, in the succeeding examples, we are going to find independent decompositions of reaction networks.
}
%\begin{remark}
%Instead of the stoichiometric matrix of the CRN, we will use its transpose in our analysis as it is easier for us to study in that manner.
%\end{remark}

%\subsection{CRNs with two reactions}

%\subsubsection{Basic Observation}
\subsection{Some Basic Observations}
\label{basic:obs}
{
	In this section, we consider basic observations, in the form of propositions, about partitioning the set of reaction vectors into two classes, which is somehow easy to follow. However, this will be generalized in the next section that covers the number of classes greater than two.
	We begin with the following easy proposition.
}
\begin{proposition}
	Suppose a CRN has two reactions. If one reaction vector is a multiple of the other, in particular, the reactions are reversible pairs, then, there is no nontrivial independent decomposition. Otherwise, the CRN has a nontrivial decomposition.
\end{proposition}

{
	\begin{proof}
		Let $\mathscr{N}$ be a CRN. Since $\mathscr{N}$ has two reactions then there are only two decompositions, the trivial, and the one where each subnetwork has only a single reaction. For the latter, the corresponding stoichiometric subspaces yields $1=s\ne s_1+s_2=1+1=2$ since one reaction vector is a multiple of the other, and hence, result to non-independence.
	\end{proof}
}

\begin{example}
	Consider the network with the following reactions 
	\[\begin{array}{l}
		{R_1}:{X_1} \to {X_2}\\
		{R_2}:{X_2} \to {X_1}
	\end{array}.\]
	Since $R_1$ and $R_2$ are reversible pairs, then the independent decomposition is trivial.
\end{example}

\begin{example}
	Consider the following CRN: 
	\[\begin{array}{l}
		{R_1}:{2X_1} \to {X_2}\\
		{R_2}:{X_2} \to {X_3}
	\end{array}.\]
	Since the corresponding reaction vectors are not multiple of each other, then a nontrivial independent decomposition exists and is given by $\{\{R_1\},\{R_2\}\}$.
\end{example}

We now consider more general statements with the following propositions:

\begin{proposition}
	Suppose $v_1 \in P_1$ and $v_2 \in P_2$, where reaction vectors $v_1$ and $v_2$ are multiples of each other, and $P_1$ and $P_2$ are distinct classes. Then, the { decomposition of the network induced by partitioning the reaction vectors} is not independent. In particular, in an independent decomposition, reversible pairs must belong to the same class.
\end{proposition}

{
	For a sketch of proof, consider a partition of the set of reaction vectors $P=P_1 \cup P_2$. We have $\dim {\rm span} (P)<\dim {\rm span} (P_1)+\dim {\rm span} (P_2)$
	since $v_1,v_2 \in P$ while $v_1 \in P_1$ and $v_2 \in P_2$ such that $v_1=kv_2$ for some scalar $k$.
}

{
	\begin{example}
		Suppose we have a CRN with the following reactions:
		\[\begin{array}{l}
			{R_1}:0 \to {X_1}\\
			{R_2}:{X_1} \to {X_2}\\
			{R_3}:{X_2} + {X_3} \to {X_1} + {X_3}\\
			{R_4}:{X_2} \to {X_3}
		\end{array}\]
		We choose a decomposition $\left\{ {\left\{ {{R_1},{R_2}} \right\},\left\{ {{R_3},{R_4}} \right\}} \right\}$. Note that the reaction vectors of $R_2$ and $R_3$ are ${X_2} - {X_1}$ and ${X_1} - {X_2}$, respectively. Since, ${X_2} - {X_1} =  - \left( {{X_1} - {X_2}} \right)$, and the two reaction vectors belong to two different classes, then the decomposition is not independent.
	\end{example}
}

\begin{proposition}
	Suppose $v_1 \in P_1$ and $v_2 \in P_2$, where $\{v_1,v_2\}$ is a linearly independent subset of reaction vectors, and $P_1$ and $P_2$ are distinct classes. If there exists another reaction vector $v_3=a_1 v_1 + a_2 v_2$ such that $a_1$ and $a_2$ are both nonzero, then the decomposition is not independent.
\end{proposition}

{
	Sketching a proof, we consider a partition of the set of reaction vectors $P=P_1 \cup P_2$ and let $v_1 \in P_1$ and $v_2 \in P_2$. Without loss of generality, assume $v_3=a_1 v_1 + a_2 v_2 \in P_1$ for scalars $a_1$ and $a_2$. Since $v_1,v_2,v_3 \in P$ but $v_1,v_3 \in P_1$ and $v_2 \in P_2$, $\dim {\rm span} (P)<\dim {\rm span} (P_1)+\dim {\rm span} (P_2)$. Non-independence follows directly.
}

{
	\begin{example}
		Consider the following reaction network.
		\[\begin{array}{l}
			{R_1}:0 \to A\\
			{R_2}:A \to B\\
			{R_3}:B \to 0
		\end{array}\]
		Let $P$ be the set of all reaction vectors, and consider a decomposition yielding two partitions $P_1$ and $P_2$ such that the reaction vector for $R_1$: $A \in P_1$ and the reaction vector for $R_3$: $-B \in P_2$. We are left with the reaction vector of $R_2$: $B-A$, a linear combination of the two previous reaction vectors. With this set-up, one cannot find an independent decomposition. We can further verify this in the following argument. If $B-A$ belongs to $P_1$, then $\dim {\rm span} (P_1)=2$. Now, $\dim {\rm span} (P_1)+\dim {\rm span} (P_2)=2+1=3>2=\dim {\rm span} (P)$. We can check analogously when $B-A$ belongs to $P_2$. In any case, the decomposition is not independent.
	\end{example}
}

\begin{remark}
	An implication of the results given above is that in an independent decomposition, the reaction vectors that are multiple of each other must belong to the same class.
\end{remark}

\subsection{The Main Theorem, Method, and Applications}
We begin by relating the decomposition of a network with respect to its set of reactions and the corresponding set of reaction vectors.

\noindent Let $\mathscr{R}=\{R_1,\ldots, R_m\}$ be the set of reactions  in a CRN $(\mathscr{S}, \mathscr{C}, \mathscr{R})$. Recall that for each reaction $R_i \in \mathscr{R}$ of the form $C_{i1} \to C_{i2}$,
the \emph{corresponding reaction vector} of $R_i$, denoted by ${\bf R}_i$, is defined by $${\bf R}_i:=C_{i2}-C_{i1}.$$ A \emph{decomposition} of $\mathscr{R}$ is a collection $\mathscr{P}^{\prime}=\{ P_1^{\prime},\ldots, P_m^{\prime}\}$ of subsets of $\mathscr{R}$ 
such that 
\begin{center} $\bigcup P_i^{\prime} = \mathscr{R}$ and $P_i^{\prime} \cap P_j^{\prime} = \emptyset $ for all $i \neq j$. 
\end{center}
Denote by ${\bf R}$ the set of reaction vectors of $\mathscr{R}$. The decomposition of ${\bf R}$, denoted by $\mathscr{P}$, is a collection $\mathscr{P}=\{ P_1,\ldots, P_m\}$ of subsets of ${\bf R}$ 
such that ${\bf R}_k \in P_i$ if and only if the corresponding $R_k^{\prime} \in P_i^{\prime}$.
{
	Note that in Section \ref{basic:obs}, we consider two classes $P_1$ and $P_2$. In this section, however, we consider the generalized partitioning of the set of reaction vectors into $m$ classes.
}
Let $p$ be the dimension of the span of ${\bf R}$, or equivalently, 
\begin{equation}p= \mbox{dim}\left(\sum_{i=1}^m \mbox{span}(P_i)\right).
\end{equation}
The dimension of a sum of subspaces is always at most the sum of the dimensions of these subspaces \cite[Theorem 1.14]{Roman}, and so we have
\begin{equation}
	\label{independent} p \leq\sum_{i=1}^m\mbox{dim} \left( \mbox{span}(P_i)\right).
\end{equation}
We say that the decomposition $\mathscr{P}$ is \emph{independent} if we have equality in (\ref{independent}), otherwise, we say that $\mathscr{P}$ is \emph{dependent}.
We also say that the decomposition $\mathscr{P}^{\prime}$ of $\mathscr{R}$ is independent if and only if the corresponding decomposition $\mathscr{P}$ is independent.

We wish to determine whether an independent decomposition of $\mathscr{R}$ exists. In essence, we only need to look for an independent decomposition
of ${\bf R}$. Let ${\bf R}=\{{\bf R}_1,\ldots, {\bf R}_m\}$. Observe that the existence of an independent decomposition of ${\bf R}$
is invariant under reindexing of the elements of ${\bf R}$. Hence, without loss of generality, we assume that $\{{\bf R}_1,\ldots, {\bf R}_p\}$ is linearly independent. The following is central in the
statement of our main theorem.

\begin{definition}
	\label{def:coordinate:graph}
	Let ${\bf R}=\{{\bf R}_1,\ldots, {\bf R}_m\}$ be a set of vectors such that the span of ${\bf R}$ is of dimension $p$, and suppose that $\{{\bf R}_1,\ldots,{\bf R}_p\}$ is linearly independent.
	The coordinate graph of ${\bf R}$ is the (undirected) graph $G=(V,E)$ with vertex set $V=\{v_1,\ldots, v_p\}$ and edge set $E$ such that
	$(v_i,v_j)$ is an edge in $E$ if and only if there exists $k >p$ with ${\bf R}_k=\displaystyle \sum_{j=1}^p a_j {\bf R}_j$ and both $a_i$ and $a_j$ are nonzero.
\end{definition}

We recall some basic definitions in graph theory.
\begin{definition} Let $G$ be a graph with vertex set $V$ and edge set $E$.
	\begin{enumerate}
		\item A subgraph $G_0=(V_0,E_0)$ is a graph with vertex set $V_0$ and edge set $E_0$ such that
		$V_0 \subseteq V$ and $E_0 \subseteq E$.
		\item A path in $G$ is a finite sequence of edges that joins two vertices. 
		\item $G$ is connected if every two vertices in $V$ is joined by a path in $G$.
		\item A connected component of $G$ is a maximal connected subgraph of $G$. We have that $G$ is connected if and only if it only has
		one connected component. Also, if $G_1$ and $G_2$ are distinct components of $G$, then whenever $v_i $ is a vertex in $G_i$ for $i=1,2$,
		then there is no path connecting $v_1$ and $v_2$.
	\end{enumerate}
\end{definition}

{
	\begin{example}
		\label{coordinate:example}
		Take the CRN with corresponding reaction vectors in Table \ref{network:coordinate:example}.
		\begin{table}
			\centering
			\caption{CRN in Example \ref{coordinate:example}}
			\label{network:coordinate:example}       % Give a unique label
			\begin{tabular}{lr}
				reaction  & reaction vector \\
				\noalign{\smallskip}\hline\noalign{\smallskip}
				$0 \to X_1$ & $X_1$\\
				$X_1 \to X_2$ & $X_2-X_1$\\
				$X_2 \to 0$ & $-X_2$\\
				$X_3+X_2 \to X_4+X_2$ & $X_4-X_3$\\
				$X_4 \to X_3$ & $X_3-X_4$\\
				\noalign{\smallskip}\hline\noalign{\smallskip}
			\end{tabular}
		\end{table}
		Consider this linearly independent set: $V:=\{X_1,-X_2,X_3-X_4\}$. Let $v_1=X_1$, $v_2=-X_2$, and $v_3=X_3-X_4$. The coordinate graph is given by $G=(V,E)$ where $V=\{v_1,v_2,v_3\}$. In Definition \ref{def:coordinate:graph}, $e=(v_1,v_2)$ is an edge since there is a reaction vector $X_2-X_1$ which is a linear combination of $v_1$ and $v_2$. On the other hand, we cannot get an edge with $v_3$ as vertex since we cannot express $v_3$ in terms of linear combination of at least two vectors or vertices in $V$. Therefore, the coordinate graph is given in Fig \ref{coord:ex}. Note that the graph is not connected and there are two distinct components.
		\begin{figure}
			\begin{center}
				\includegraphics[width=15cm,height=5cm,keepaspectratio]{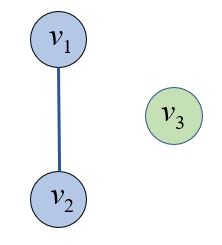}
				\caption{Coordinate graph in Example \ref{coordinate:example}}
				\label{coord:ex}
			\end{center}
		\end{figure}
	\end{example}
}

The following theorem gives a necessary and sufficient condition for the existence of an independent decomposition of a set of vectors.
\begin{theorem} Let  ${\bf R}$ be a finite set of vectors. An independent decomposition of ${\bf R}$ exists if and only if 
	the coordinate graph of ${\bf R}$ is not connected. 
\end{theorem}
\begin{proof}
	Let ${\bf R}=\{{\bf R}_1,\ldots, {\bf R}_m\}$ be a set of vectors and let $p$ be the dimension of the span of ${\bf R}$, and suppose that $\{{\bf R}_1,\ldots,{\bf R}_p\}$ is linearly independent. 
	Let  $G=(V,E)$ have vertex set $V$ and edge set $E$.
	
	Suppose that the coordinate graph of $G$ is not connected, that is, $G$ has at least two connected components. Let
	$V=\{v_1,\ldots, v_m\}$. We do the case when $G$ has only two connected components
	say $G_1=(V_1,E_1)$ and $G_2=(V_2,E_2)$. We also reindex the vectors in $R$, if necessary, so that $V_1=\{v_1,\ldots, v_s\}$ and $V_2=\{v_{s+1},\ldots, v_p\}$.
	We now form the independent decomposition of ${\bf R}$. We start by setting $P_1$ to contain ${\bf R}_{1},\ldots, {\bf R}_s$ and by setting $P_2$ to contain ${\bf R}_{s+1},\ldots, {\bf R}_p$.
	
	For $k>p$, we  now set a rule to put ${\bf R}_{k} =\displaystyle \sum_{i=1}^p a_{i} {\bf R}_i$ in either $P_1$ and $P_2$. Observe that either 
	\begin{center} $a_{1}=\cdots=a_s=0$ or $a_{s+1}=\cdots=a_p=0$,
	\end{center}
	since if $a_i \neq 0$ for some $i \leq s$
	and $a_j \neq 0 $ for some $j>s$, then $(v_i,v_j) \in E$, and this contradicts the assumption that $G_1$ and $G_2$ are distinct components of $G$.
	If it is the former, we set $P_2$ to contain ${\bf R}_{k}$, and otherwise, we set $P_1$ to contain ${\bf R}_{k}$. This gives us that any vector added to the original 
	set $\{{\bf R}_1,\ldots, {\bf R}_s\}$ in $P_1$ are linear combination of $\{{\bf R}_1,\ldots, {\bf R}_s\}$, and so the dimension of the span of $P_1$ is $s$. Similarly,
	the dimension of the span of the final $P_2$ is $p-s$.
	This tells us that 
	$\mathscr{P}=\{P_1,P_2\}$ is an independent decomposition of ${\bf R}$.
	
	Suppose that the coordinate graph of $G$ is connected. Suppose that $\mathscr{P}=\{P_1,\ldots, P_k\}$ is an independent decomposition of ${\bf R}$.
	We do the case when $k=2$ and the general case can be proven similarly. Suppose that $P_1$ has $s$ elements from $\{{\bf R}_1,\ldots, {\bf R}_p\}$ and that $s$ is maximal
	in the sense that we can no longer add elements from $\{{\bf R}_1,\ldots, {\bf R}_p\} \setminus P_1$ to $P_1$ 
	that makes $\{P_1,P_2\}$ an independent decomposition of ${\bf R}$.
	We can reindex the elements of $\{{\bf R}_1, \ldots, {\bf R}_p\}$ so that
	${\bf R}_1,\ldots, {\bf R}_s$ belong to $P_1$ and ${\bf R}_{s+1},\ldots, {\bf R}_p$ belong to $P_2$. Now, since $G$ is connected, there is an $i \leq s$ and there is a
	$j>s$ such that $(v_i,v_j)$ is an edge of $G$. This implies that there is an ${\bf R}_{t}$ for some $t>p$ such that
	${\bf R}_t= \displaystyle \sum_{l=1}^p a_l {\bf R}_l$, where $a_i$ and $a_j$ are nonzero. If ${\bf R}_t$ belongs to $P_1$, then
	${\bf R}_1,\ldots, {\bf R}_s,{\bf R}_t$ is linearly independent, and so, since $P_1$ contains ${\bf R}_1,\ldots, {\bf R}_s,{\bf R}_t$ and $P_2$ contains
	${\bf R}_{s+1},\ldots, {\bf R}_p$, we have
	$$\mbox{dim} \left( \mbox{span}(P_1)\right) + \mbox{dim} \left( \mbox{span}(P_2)\right) \geq (s+1)+(p-s)=p+1.$$
	We arrive at the same conclusion when ${\bf R}_t$ belongs to $P_2$. This implies that we cannot have an independent decomposition of ${\bf R}$, as desired.
\end{proof}

Let ${\bf R}=\{{\bf R}_1,\ldots, {\bf R}_m\}$ be a set of vectors such that the span of ${\bf R}$ is of dimension $p$, and suppose that ${\bf S}=\{{\bf R}_1,\ldots,{\bf R}_p\}$ is linearly independent.
Suppose that the coordinate graph of $G$ is not connected, and for the sake of illustration, assume that $G$ has only two components $G_1$ and $G_2$.
We take the elements of $S$ corresponding to the vertices of $G_1$ with their linear combinations in ${\bf R}$ and form $P_1$, and we take the elements of ${\bf S}$ 
corresponding to the vertices of $G_2$ with their linear combinations in ${\bf R}$ and form $P_2$. The proof above ensures that $\mathscr{P}=\{P_1,P_2\}$ is
a partition of ${\bf R}$.

We now consider the following detailed method to obtain an independent decomposition for a CRN, if it exists.

\begin{remark}
	For ease of notation, we also use ${R}_i$ for reaction vector ${\bf R}_i$. After all, reaction vector ${\bf R}_i$ is identified with reaction ${R}_i$. We emphasize that reaction vectors may not be unique, as two reactions may have the same reaction vectors.
\end{remark}

\noindent {\bf Method of Finding Independent Decompositions of CRNs}

Given a CRN, we consider a method of getting an independent decomposition, if it exists, using the following steps:

\begin{itemize}
	\item[1.] Get the transpose of the stoichiometric matrix $N$, which we denote by $N^T$.
	\item[2.] Find a maximal linearly independent set of vectors say $\{R_{i_1},R_{i_2}\ldots, R_{i_p}\}$, which forms a basis for the row space of $N^T$.
	\item[3.] Construct the vertex set of the coordinate graph $G=(V,E)$ of $R$ by representing each $R_{i_j}$ as vertex $v_i$.
	\item[4.] For each vector $R_k$ distinct from 
	the elements of $\{R_{i_1},R_{i_2}\ldots, R_{i_p}\}$,
	write $R_k= \displaystyle \sum_j a_{k,j} R_{i_j}$.
	For each pair $a_{k,j_1}$ and $a_{k,j_2}$
	in the preceding sum,
	with both coefficients being nonzero, we add
	the edge $(v_{j_1},v_{j_2})$ to $E$.
	\item[5.] If the formed coordinate graph $G$
	is connected, then there is no nontrivial
	decomposition for $R$. Otherwise, the
	reaction vectors corresponding to vertices
	belonging in the same connected component, together with the reaction vectors in their span
	constitute a partition of $R$ in the
	independent decomposition of
	$R$.
\end{itemize}

{
	We consider the following examples to illustrate our method of finding independent decompositions of reaction networks, if they exist.
}

\begin{example}
	We consider the generalized mass action model of anaerobic fermentation pathway of \it{Saccharomyces cerevisiae} \cite{curto,voit}.
	The CRN is provided in \cite{AJLM2017}.
	\[ \begin{array}{lll}
		R_1: X_2 \to X_1 + X_2 & \ \ \ &  R_8: X_3 + X_5  \to X_4 + X_5\\
		R_2: X_1 + X_5  \to X_2 + X_5  &  \ \ \ & R_9: X_3 + X_5 \to X_3 + 2X_5 \\
		R_3: 2X_5 + X_1  \to X_5 + X_1 &  \ \ \ & R_{10}: X_3 + X_4 + X_5 \to X_4 + X_5\\
		R_4: X_2 + X_5  \to X_3 + X_5  &  \ \ \ & R_{11}: X_3 + X_4 + X_5 \to X_3 + X_5  \\
		R_5: 2X_5 + X_2  \to X_5 + X_2 & \ \ \  & R_{12}: X_3 + X_4 + X_5  \to X_3 + X_4 + 2X_5\\
		R_6: X_2 + X_5  \to X_5  &  \ \ \ & R_{13}: 2X_5 \to X_5 \\
		R_7: X_2 + X_5  \to X_2  &  \\
	\end{array}\]
	The following matrix can be seen as the transpose of the stoichiometric matrix.
	\begin{equation}
		\label{GMA}
		\begin{bmatrix} 1 & 0 & 0 & 0 & 0 \\ -1 & 1 & 0 & 0 & 0 \\ 0 & 0 & 0 & 0 & -1 \\ 0 & -1 & 1 & 0 & 0 \\ 0 & 0 & 0 & 0 & -1 \\ 0 & -1 & 0 & 0 & 0 \\ 
			0 & 0 & 0 & 0  & -1 \\ 0 & 0 & -1 & 1 & 0 \\0 & 0 & 0 & 0 & 1 \\ 0 & 0 & -1 & 0 & 0 \\ 0 & 0 & 0 & -1 & 0 \\ 0 & 0 & 0 & 0 & 1 \\ 0 & 0 & 0 & 0 & -1 
		\end{bmatrix}
	\end{equation}
	
	Let $R_i$ be the $i$th row of (\ref{GMA}). Observe that the rows $R_1, R_2, R_3, R_4, R_8$ form a basis for the row space of (\ref{GMA}). We also have the following relations for the rows
	of (\ref{GMA}).
	\begin{enumerate}
		\item $R_5= R_3$
		\item $R_6=-R_2-R_1$
		\item $R_7=R_3$
		\item $R_9=-R_3$
		\item $R_{10}=-R_4-R_2-R_1$
		\item $R_{11}=-R_8-R_4-R_2-R_1$
		\item $R_{12}=-R_3$
		\item $R_{13}=R_3$
	\end{enumerate}
	
	We now represent $R_1,R_2,R_3,R_4,R_8$ respectively by
	$v_1,v_2,v_3,v_4,v_5$. The linear combination above, then contributes the following edges to the coordinate graph $G$ of $R$ For an illustration, please refer to Figure \ref{sacc:fig}.
	
	\begin{enumerate}
		\item $R_6$: $(v_1,v_2)$
		\item $R_{10}$: $(v_1,v_2)$, $(v_1,v_4)$, $(v_1,v_4)$
		\item $R_{11}$: $(v_i,v_j)$ where $i,j \in \{ 1,2,4,5\}$
	\end{enumerate}
	
	\begin{figure}
		\begin{center}
			\includegraphics[width=15cm,height=5cm,keepaspectratio]{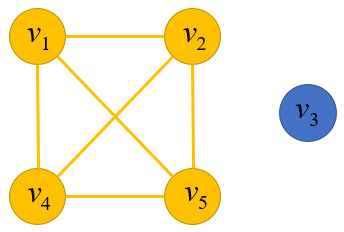}
			\caption{Coordinate graph for the yeast pathway CRN}
			\label{sacc:fig}
		\end{center}
	\end{figure}
	
	This gives us two maximal components for $G$ which induces the following independent decomposition of $R$.
	\begin{enumerate}
		\item $P_1=\{R_1,R_2,R_4,R_6,R_8,R_{10},R_{11}\}$
		\item $P_2=\{R_3,R_5,R_7,R_9,R_{12},R_{13}\}$
	\end{enumerate}
	Thus, the CRN of the given model has a nontrivial independent decomposition $\{P_1,P_2\}$.
\end{example}

\begin{example}
	Consider a particular biological system, which is a metabolic network with one positive feedforward and a negative feedback obtained from the published work of Sorribas et al. \cite{SHVA2007}. The following CRN corresponds to the metabolic network with $X_5$ as independent variable.
	
	\[ \begin{array}{lll}
		R_1: 0 \to X_1 & \ \ \ &  R_4: X_1+X_2  \to X_1+X_4\\
		R_2: X_1+X_3  \to  X_3+X_2  &  \ \ \ & R_5: X_3 \to 0 \\
		R_3: X_2  \to X_3 &  \ \ \ & R_6: X_4 \to 0\\
	\end{array}\]
	
	\begin{equation}
		\left[ {\begin{array}{*{20}{c}}
				1&0&0&0\\
				{ - 1}&1&0&0\\
				0&{ - 1}&1&0\\
				0&{ - 1}&0&1\\
				0&0&{ - 1}&0\\
				0&0&0&{ - 1}
		\end{array}} \right]
		\label{sorr}
	\end{equation}
	
	Let $R_i$ be the $i$th row of (\ref{sorr}). Observe that the rows $R_1, R_4, R_5, R_6$ form a basis for the row space of the matrix above. We also have the following relations for the rows
	of (\ref{sorr}).
	\begin{enumerate}
		\item $R_2= -R_1-R_4-R_6$
		\item $R_3=-R_4-R_5+R_6$
	\end{enumerate}
	
	We now represent $R_1, R_4, R_5, R_6$ respectively by
	$v_1,v_2,v_3,v_4$. The linear combination above, then contributes the following edges to the coordinate graph $G$ of $R$.
	
	\begin{enumerate}
		\item $R_2$: $(v_1,v_2),(v_1,v_4)$
		\item $R_{3}$: $(v_2,v_3)$, $(v_2,v_4)$
	\end{enumerate}
	
	Thus, there must only be one component for G. Therefore, the decomposition is trivial.

\end{example}

\subsection{Further Applications to Influenza Virus Models}

We now apply our results to influenza virus models integrated with deficiency theorems and the Feinberg Deficiency Theorem (FDT).

\begin{example}
	We consider the Baccam Model with three variables: uninfected (susceptible) target cells (T), infected
	cells (I) and infectious-viral titer (V) \cite{baccam}. {A diagram is given in Fig. \ref{baccam:fig}. Here, uninfected target cells become infected, and infected cells die spontaneously at rate of $\beta$ and $\delta$, respectively.
		In addition, virus proliferates and dies at the rate of $p$ and $c$, respectively.} With respect to the network, the following are the reactions \cite{virus}.
	\begin{align*}
		{R_1}&:{T}+V \to {I} + {V}\\
		{R_2}&:{I} \to {0} \\
		{R_3}&: I \to {I} + {V}\\
		{R_4}&:{V} \to {0} 
	\end{align*}
	
	\begin{figure}
		\begin{center}
			\includegraphics[width=15cm,height=5cm,keepaspectratio]{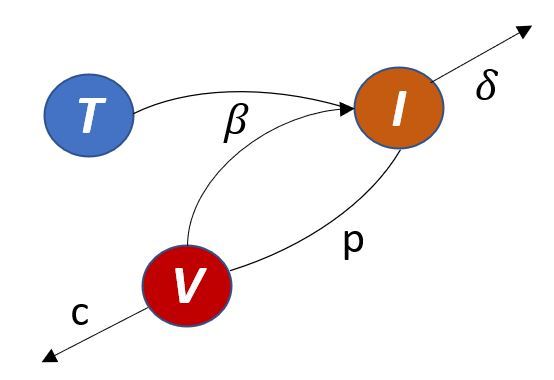}
			\caption{Diagram for Baccam Model adapted from \cite{virus}}
			\label{baccam:fig}
		\end{center}
	\end{figure}
	
	We can obtain the following independent decomposition of $$\mathscr{N}=\{R_1,R_2,R_3,R_4\}:$$ $\{\mathscr{N}_1,\mathscr{N}_2\}$ where $\mathscr{N}_1=\{R_1,R_2\}$ and $\mathscr{N}_2=\{R_3,R_4\}$.
	\begin{table}
		\centering
		\caption{Network numbers for Baccam CRN}
		\label{network:baccam}       % Give a unique label
		\begin{tabular}{lrrr}
			\noalign{\smallskip}\hline\noalign{\smallskip}
			& $\mathscr{N}$  & $\mathscr{N}_1$ & $\mathscr{N}_2$\\
			\noalign{\smallskip}\hline\noalign{\smallskip}
			\# species & 3 & 3 & 2\\
			\# complexes & 5 & 4 & 4\\
			\# reactions & 4 & 2 & 2\\
			\# irreversible reactions & 4 & 2 & 2\\
			\# linkage classes & 1 & 2 &2\\
			rank of network & 3 & 2 & 1\\
			deficiency & 1 & 0 &1\\
			\noalign{\smallskip}\hline\noalign{\smallskip}
		\end{tabular}
	\end{table}
	We consider the first subnetwork. Note that its deficiency is zero. From the Deficiency Zero Theorem (DZT), as the subnetwork is not weakly reversible, it has no capacity for multistationarity.
	Since the decomposition is independent, from the FDT (i.e., Theorem \ref{feinberg:decom:thm}), the intersection of the set of positive steady states of the systems equals the set of positive steady states of the whole network.
	Therefore, the whole network (of deficiency one) does not also have the capacity for multistationarity, as the following statement in DZT is satisfied:
	If the network is not weakly reversible, then for arbitrary kinetics, the differential equations for the corresponding reaction system cannot admit
	a positive steady state.
	In this particular example, our result is strong in the sense that it considers arbitrary kinetics.
	\label{ex:baccam}
\end{example}

\begin{example}
	We now consider the Baccam Model with delayed virus production \cite{baccam,miao}. The following are the reactions \cite{virus}.
	\begin{align*}
		{R_1}&:{T}+V \to {I_1} + {V}\\
		{R_2}&:{I_1} \to {I_2} \\
		{R_3}&:{I_2} \to {0} \\
		{R_4}&: I_2 \to {I_2} + {V}\\
		{R_5}&:{V} \to {0} 
	\end{align*}
	We can obtain the following independent decomposition of $$\mathscr{N}=\{R_1,R_2,R_3,R_4,R_5\}:$$ $\{\mathscr{N}_1,\mathscr{N}_2\}$ where $\mathscr{N}_1=\{R_1,R_2,R_3\}$ and $\mathscr{N}_2=\{R_4,R_5\}$.
	\begin{table}
		\begin{center}
			\caption{Network numbers for Baccam CRN (with delayed virus production)}
			\label{network:baccamdelay}       % Give a unique label
			\begin{tabular}{lrrr}
				\noalign{\smallskip}\hline\noalign{\smallskip}
				& $\mathscr{N}$  & $\mathscr{N}_1$ & $\mathscr{N}_2$\\
				\noalign{\smallskip}\hline\noalign{\smallskip}
				\# species & 4 & 4 & 2\\
				\# complexes & 7 & 5 & 4\\
				\# reactions & 5 & 3 & 2\\
				\# irreversible reactions & 5 & 3 & 2\\
				\# linkage classes & 2 & 2 &2\\
				rank of network & 4 & 3 & 1\\
				deficiency & 1 & 0 &1\\
				\noalign{\smallskip}\hline\noalign{\smallskip}
			\end{tabular}
		\end{center}
	\end{table}
	We consider the first subnetwork. Note that its deficiency is zero. From the Deficiency Zero Theorem, as the subnetwork is not weakly reversible, it has no capacity for multistationarity. It follows from the FDT that the whole network does not also have the capacity for multistationarity.
	\label{ex:baccam:delay}
\end{example}

\begin{remark}
	If an independent decomposition of a network has a non-weakly reversible subnetwork with deficiency zero, then by the DZT and the FDT, for arbitrary kinetics, the whole system cannot admit a positive steady state. This happens for the case of Examples \ref{ex:baccam} and \ref{ex:baccam:delay}.
\end{remark}

\begin{example}
	We now consider the Handel model with the following variables: uninfected cells (U), latently infected cells
	(E), productively infected cells (I), dead cells (D), free viruses (V), innate immune response (F), and
	adaptive immune response (X) \cite{handel}. The following are the reactions with the given kinetics \cite{virus}.
	\begin{align*}
		{R_1}&:D \to U & \lambda D\\
		{R_2}&:U+V \to E+V & bUV\\
		{R_3}&:E \to I & gE\\
		{R_4}&:I \to D & dI\\
		{R_5}&:I \to I+V & \dfrac{pI}{1+\kappa F}\\
		{R_6}&:V \to 0 & cV\\
		{R_7}&:U+V \to U & \gamma bUV\\
		{R_8}&:V+X \to X & kVX\\
		{R_9}&:V \to V+F & wV\\
		{R_{10}}&:F \to 0 & \delta F\\
		{R_{11}}&:V \to V+X & fV\\
		{R_{12}}&:X \to 2X & rX
	\end{align*}
	In addition, the following is the set of corresponding ODEs \cite{handel}.
	\begin{align*}
		{\text{uninfected cells \ \ }}U'&= \lambda D - bUV\\
		{\text{latently infected cells \ \ }}E' &= bUV - gE\\
		{\text{productively infected cells \ \ }}I' &= gE - dI\\
		{\text{dead cells \ \ }}D' &= dI - \lambda D\\
		{\text{free virus \ \ }}V' &= \dfrac{{pI}}{{1 + \kappa F}} - cV - \gamma bUV - kVX\\
		{\text{innate immune response \ \ }}F' &= wV - \delta F\\
		{\text{adaptive immune response \ \ }}X' &= fV + rX
	\end{align*}
	
	We use the step-by-tep procedure for getting an independent decomposition, if it exists.
	\begin{equation}
		\label{handel}
		\begin{bmatrix} -1 & 1 & 0 & 0 & 0 & 0 & 0 \\  0 & -1 & 0 & 0 & 0 & 0 & 1 \\  0 &0 & 0 & 0 & 0 & 1 & -1 \\   1 & 0 & 0 & 0 & 0 &-1 & 0 \\
			0 & 0 & 1 & 0 & 0 & 0 & 0 \\ 0 & 0 & -1 & 0 & 0 & 0 & 0 \\ 0 & 0 & -1 & 0 & 0 & 0 & 0   \\ 0 & 0 & -1 & 0 & 0 & 0 & 0 \\
			0 &0 & 0 & 1 & 0 & 0 & 0 \\ 0 &0 & 0 & -1 & 0 & 0 & 0 \\ 0 &0 &0 & 0 & 1 & 0 & 0  \\  0 &0 &0 & 0 & -1 & 0 & 0 \end{bmatrix}
	\end{equation}
	
	Let $R_i$ be the $i$th row of (\ref{handel}). Observe that the rows $R_1,R_2,R_3, R_5, R_9, R_{11}$ form a basis for
	the rowspace of (\ref{handel}). We have the following relations for the rows of (\ref{handel}):
	\begin{enumerate}
		\item $R_4=-R_1-R_2-R_3$
		\item $R_6=-R_5$
		\item $R_7=-R_5$
		\item $R_8=-R_5$
		\item $R_{10}=-R_9$
		\item $R_{12}=-R_{11}$
	\end{enumerate}
	
	We now represent $R_1,R_2,R_3, R_5, R_9, R_{11}$ respectively by $v_1,v_2,v_3,v_4,v_5,v_6$. The linear combination above, then contributes the following edges
	to the coordinate graph $G$ of $R$.
	
	\begin{enumerate}
		\item $R_4$: $(v_1,v_2)$, $(v_1,v_3)$, $(v_2,v_3)$ (see Figure \ref{handel:fig})
	\end{enumerate}
	
	\begin{figure}
		\begin{center}
			\includegraphics[width=17cm,height=6cm,keepaspectratio]{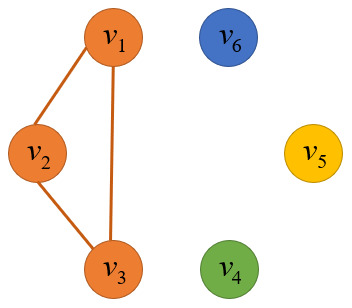}
			\caption{Coordinate graph for the Handel CRN}
			\label{handel:fig}
		\end{center}
	\end{figure}
	
	This gives us four connected components for $G$ which induces the following independent decomposition of $R$.
	
	\begin{enumerate}
		\item $P_{1}=\{ R_1,R_2,R_3,R_4\}$
		\item $P_2=\{R_5,R_6,R_7,R_8\}$
		\item $P_3=\{R_9,R_{10}\}$
		\item $P_4=\{R_{11},R_{12}\}$
	\end{enumerate}
	
	Hence, we have $\mathscr{N}=\{R_i:i=1,...,12\}$ with the following independent decomposition: $\{\mathscr{N}_1,\mathscr{N}_2,\mathscr{N}_3,\mathscr{N}_4\}$ where $\mathscr{N}_1=\{R_1,R_2,R_3,R_4\}$, $\mathscr{N}_2=\{R_5,R_6,R_7,R_8\}$, $\mathscr{N}_3=\{R_9,R_{10}\}$, and $\mathscr{N}_4=\{R_{11},R_{12}\}$.
	\begin{table}
		\centering
		\caption{Network Numbers for Handel CRN}
		\label{network:handel}       % Give a unique label
		\begin{tabular}{lrrrrr}
			\noalign{\smallskip}\hline\noalign{\smallskip}
			& $\mathscr{N}$  & $\mathscr{N}_1$ & $\mathscr{N}_2$& $\mathscr{N}_3$ & $\mathscr{N}_4$\\
			\noalign{\smallskip}\hline\noalign{\smallskip}
			\# species & 7 & 5 & 4 &2 &2\\
			\# complexes & 14 & 6 & 8 & 4 & 4\\
			\# reactions & 12 & 4 & 4 & 2 & 2\\
			\# irreversible reactions & 12 & 4 & 4 & 2 & 2\\
			\# linkage classes & 2 & 2 &4 &2 & 2\\
			rank of network & 6 & 3 & 1 & 1 & 1\\
			deficiency & 6 & 1 &3 & 1 & 1\\
			\noalign{\smallskip}\hline\noalign{\smallskip}
		\end{tabular}
	\end{table}
	We consider the subnetwork $\mathscr{N}_4$. Note that its deficiency is 1. 
	We can actually use the Multistationarity Algorithm in \cite{fortun,hmr2019,ji}. For other kinetics, such as Hill-type, one may use the method given in \cite{hernandez:jomc2}.
	We will not show the details here as it requires a lot of preliminaries. Since the preliminary analysis is not satisfied, it has no capacity for multistationarity. It follows from the FDT that the whole network does not also have the capacity for multistationarity.
\end{example}

{
	\subsection{Additional Examples}
	We now consider two other examples. The first example is the use of our results for indeed a large network. The CRN in Example \ref{large:network} has 42 reactions. The other example provides a case when you can show that when a vector is a steady state of each system under independent decomposition, it is also a steady state of the whole system with the underlying network.
	
	\begin{example}
		\label{large:network}
		Consider a generalized mass action model of purine metabolism in man \cite{arceo2015,voit}. We take into consideration the total representation with the following reactions provided in the supplementary materials of Arceo et al. \cite{AJLM2017}.\\
		${R_1}:{X_1} + {X_{17}} + {X_4} + {X_8} + {X_{18}} \to 2{X_1} + {X_4} + {X_8} + {X_{18}}$\\
		${R_2}:{X_1} \to 0$\\
		${R_3}:{X_1} + {X_2} + {X_4} + {X_8} + {X_{18}} \to 2{X_2} + {X_4} + {X_8} + {X_{18}}$\\
		${R_4}:{X_1} + {X_4} + {X_6} \to 2{X_4} + {X_6}$\\
		${R_5}:{X_1} + {X_4} + {X_6} \to 2{X_4} + {X_1}$\\
		${R_6}:{X_1} + {X_2} + {X_{13}} \to 2{X_2} + {X_{13}}$\\
		${R_7}:{X_1} + {X_2} + {X_{13}} \to {X_1} + 2{X_2}$\\
		${R_8}:{X_1} + {X_8} + {X_{15}} \to 2{X_8} + {X_{15}}$\\
		${R_9}:{X_1} + {X_8} + {X_{15}} \to 2{X_8} + {X_1}$\\
		${R_{10}}:{X_2} + {X_4} + {X_8} + {X_{18}} \to {X_3} + {X_4} + {X_8} + {X_{18}}$\\
		${R_{11}}:{X_2} + {X_7} + {X_8} \to 2{X_7} + {X_8}$\\
		${R_{12}}:{X_2} + {X_{18}} \to {X_{13}} + {X_{18}}$\\
		${R_{13}}:{X_4} + {X_8} + {X_{18}} \to {X_2} + {X_8} + {X_{18}}$\\
		${R_{14}}:{X_2} + {X_4} + {X_8} + {X_7} \to 2{X_2} + {X_4} + {X_7}$\\
		${R_{15}}:{X_3} + {X_4} \to 2{X_4}$\\
		${R_{16}}:{X_{11}} \to {X_4}$\\
		${R_{17}}:{X_5} \to {X_4}$\\
		${R_{18}}:{X_4} + {X_5} \to 2{X_5}$\\
		${R_{19}}:{X_4} + {X_9} + {X_{10}} \to 2{X_9} + {X_{10}}$\\
		${R_{20}}:{X_4} \to {X_{13}}$\\
		$R_{21}: X_4 + X_8 \to X_8 + X_{11}$\\
		$R_{22}: X_4 + X_8 \to X_4 + X_{11}$\\
		$R_{23}: X_5 \to X_6$\\
		$R_{24}: X_6 \to 0$\\
		$R_{25}: X_4 + X_7 \to X_4 + X_8$\\
		$R_{26}: X_{11} \to X_{8}$\\
		$R_{27}: X_8 + X_{18} \to X_{15} + X_{18}$\\
		$R_{28}: X_8 + X_9 + X_{10} \to X_9 + 2X_{10}$\\
		$R_{29}: X_{12} \to X_9$\\
		$R_{30}: X_9 + X_{10} \to X_{12} + X_{10}$\\
		$R_{31}: X_9 + X_{10} \to X_{12} + X_9$\\
		$R_{32}: X_9 \to X_{13}$\\
		$R_{33}: X_{12} \to X_{10}$\\
		$R_{34}: X_{10} \to X_{15}$\\
		$R_{35}: X_{13} \to X_{14}$\\
		$R_{36}: X_{13} \to 0$\\
		$R_{37}: X_{15} \to X_{14}$\\
		$R_{38}: X_{14} \to 0$\\
		$R_{39}: X_{14} \to X_{16}$\\
		$R_{40}: X_{16} \to 0$\\
		$R_{41}: 0 \to X_{17}$\\
		$R_{42}: 0 \to X_{18}$\\
		We let $\mathscr{N}$ to be the network with 42 reactions and consider $\mathscr{N}'=\{R_{42}\}$. We can show that the decomposition resulting to $\mathscr{N}'$ and its complement with respect to the whole network is independent. Now, the subnetwork $\mathscr{N}'$ has deficiency zero and not weakly reversible. We can use the Deficiency Zero Theorem and the FDT to analyze the network.
	\end{example}

	\begin{example}
		Consider the following reaction network $\mathscr{N}$ governed by mass action kinetics:
		\[\begin{array}{l}
			{R_1}:{X_1} + {X_2} \to {X_3}\\
			{R_2}:{X_3} + {X_4} \to {X_4} + {X_1} + {X_2}\\
			{R_3}:{X_1} + {X_3} \to {X_4} + {X_1} + {X_3}\\
			{R_4}:{X_4} + 2{X_2} \to 2{X_2}
		\end{array}\]
		The vector $(2,3,3,2)$ is a steady state of both independent subnetworks $\mathscr{N}_1=\{R_1,R_2\}$ and $\mathscr{N}_2=\{R_3,R_4\}$ with $k_1=k_2=k_4=1$ and $k_3=3$. The vector $(2,3,3,2)$ is also a steady state of the whole network with the specified rate constants.
	\end{example}
}

\section{Summary and Outlook}
\label{sec:sum}
We provided a characterization, i.e., necessary and sufficient condition for the existence of a nontrivial independent decompositions of a CRN, and a novel step-by-step method to obtain such nontrivial independent decompositions, if it exists.
Moreover, we have shown that a CRN of a popular model of anaerobic yeast fermentation pathway has a nontrivial independent decomposition while a metabolic network with one positive feedforward and a negative feedback has none.
We also analyzed the properties of steady states of reaction networks of influenza virus models, in particular, the Baccam and Handel models.
For future work, one can look into the exact number of independent decompositions of a CRN. In addition, one can consider at refinements and coarsenings of such decompositions.
An analogue for incidence independent decompositions can also be explored.

\subsection*{Acknowledgements}
	This work was funded by the UP System Enhanced Creative Work and Research Grant (ECWRG 2020). The authors thank the reviewers for their comments and suggestions that really helped in the improvement of our paper. BSH is grateful to Dr. Eduardo R. Mendoza for introducing him to decompositions in Chemical Reaction Network Theory (CRNT) and sending lecture notes about CRNT that became helpful in this work.



\appendix
\section{Nomenclature}
\subsection{List of abbreviations}
%\FloatBarrier
%\begin{table}
%\caption{caption}
%\label{tab:ap1}       % Give a unique label
%\centering
\begin{tabular}{ll}
	%\hline\noalign{\smallskip}
	%first & second \\
	\noalign{\smallskip}\hline\noalign{\smallskip}
	Abbreviation& Meaning \\
	\noalign{\smallskip}\hline\noalign{\smallskip}
	CRN& chemical reaction network\\
	CRNT& chemical reaction network theory\\
	DZT& deficiency zero theorem\\
	FDT& Feinberg decomposition theorem\\
	MAK& mass action kinetic(s)\\
	MSA& multistationarity algorithm\\
	ODE& ordinary differential equation\\
	PLK& power-law kinetic(s)\\
	SFRF& species formation rate function\\
	\noalign{\smallskip}\hline
\end{tabular}
%\end{table}
%\FloatBarrier
\subsection{List of important symbols}
%\caption{caption}
%\label{tab:ap2}       % Give a unique label
%\centering
\begin{tabular}{ll}
	%\hline\noalign{\smallskip}
	%first & second \\
	\noalign{\smallskip}\hline\noalign{\smallskip}
	Meaning& Symbol \\
	\noalign{\smallskip}\hline\noalign{\smallskip}
	deficiency& $\delta$  \\
	dimension of the stoichiometric subspace& $s$   \\
	incidence matrix& $I_a$\\
	molecularity matrix& $Y$\\
	stoichiometric matrix& $N$\\
	stoichiometric subspace& $S$\\
	\noalign{\smallskip}\hline
\end{tabular}

\end{document}